\documentclass{amsart}

\usepackage{amsthm}
\usepackage{amsfonts}
\usepackage{amssymb}

\newtheorem*{thm1}{Theorem (S. Siksek)}
\newtheorem{thm}{Theorem}
\newtheorem*{lem3.1}{Lemma 3.1}
\newtheorem*{lem3.2}{Lemma 3.2}
\newtheorem*{lem3.3}{Lemma 3.3}
\newtheorem*{lem3.4}{Lemma 3.4}
\newtheorem*{lem3.4SS}{Lemma (S. Siksek)}

\newtheorem{prop}{Proposition}
\newtheorem{mylem}{Lemma}

\DeclareMathOperator{\Span}{Span}
\newcommand{\PP}{\mathbb{P}}
\newcommand{\F}{\mathbb{F}}
\newcommand{\xx}{\mathbf{x}}

\newcommand{\ch}{\mathrm{char}}

\title{Generators for Cubic Surfaces with two Skew Lines over Finite Fields}
\author{Jenny Cooley}
\address{Mathematics Institute\\University of Warwick\\ Coventry, CV4 7AL\\UK}
\email{j.a.cooley@warwick.ac.uk}
\date{\today}
\thanks{The author is supported by an Engineering and Physical Sciences Research Council (EPSRC) studentship}

\keywords{cubic surfaces, finite fields, secant and tangent, generating set}

\begin{document}

\begin{abstract}
Let $S$ be a smooth cubic surface defined over a field $K$.
As observed by Segre \cite{Segre} and Manin \cite{Ma1,Ma2}, 
there is a secant and tangent
process on $S$ that generates new $K$-rational points from old. It is
natural to ask for the size of a minimal generating set for $S(K)$.
In a recent paper, for fields $K$ with at least $13$ elements,
Siksek \cite{Siksek} showed that if $S$ contains
a skew pair of $K$-lines then $S(K)$ can be generated from
one point. In this paper we prove the corresponding version
of this result for fields $K$ having at least $4$ elements,
and slightly milder results for $\# K=2$ or $3$.
\end{abstract}

\maketitle

\section{Introduction}
Let $E$ be an elliptic curve over a field $K$. It is well-known
that, if $K$ is finite, then
 the set of $K$-rational points $E(K)$ is a finite abelian
group that is either cyclic or isomorphic to a product of two
cyclic subgroups. The group structure on $E$ is given by
the familiar secant and tangent process.

Let $S$ be a smooth cubic surface over a field $K$. There is
still a secant and tangent process that generates new points
from old points. 
This process was introduced by Segre \cite{Segre}
and studied by several authors, most notably Manin \cite{Ma1,Ma2}.
This process does not give the set of $K$-rational

points $S(K)$ a group structure. However, it is reasonable to
ask, for $K$ a field,
 whether it is still possible to generate all the $K$-rational
points from just one or two points. In a recent paper, Siksek \cite{Siksek}
shows the following. Let $S$ be a 
smooth cubic surface defined over a field $K$ having
at least $13$ elements. Suppose $S$ contains a skew pair of lines $\ell_1$,
$\ell_2$ defined over $K$. Then $S(K)$ can be generated by
just one point. The purpose of this paper is to extend the proof
of this to fields with at least $4$ elements, and prove similar
(but slightly weaker) statements over fields with $2$ or $3$ elements.

In the remainder of this introduction, we will give a precise
definition of the secant and tangent process, and state our results.
Let $K$ be field and let
$S$ be a smooth cubic surface defined over $K$. It is a well-known
classical theorem, due to Cayley and Salmon, that $S$ contains
$27$ lines defined over $\overline{K}$. Let $\ell$ be a line
not contained in $S$. Then $\ell \cdot S=P+Q+R$ where $P$, $Q$, $R$
are points on $S$, counted according to multiplicity. If $P$, $Q \in S(K)$
and $\ell$ is a $K$-line (that is, it is defined over $K$), then
$R \in S(K)$. Of course, if $P \ne Q$, then the line $\ell$ is
the secant line joining $P$, $Q$, and if $P=Q$ then $\ell$ is
a tangent line to the surface at $P$. Let $B$ be a subset of $S(K)$.
We define a sequence of sets
\[
B=B_0 \subseteq B_1 \subseteq B_2 \subseteq \cdots \subseteq S(K)
\]
inductively as follows: a point $R \in S(K)$ belongs to $B_{n+1}$ if 
and only if, either $R \in B_n$, or there are points $P$, $Q \in B_n$
and a $K$-line $\ell$ not lying on $S$ such that $\ell \cdot S=P+Q+R$. 
We let $\Span(B)= \cup_{i=0}^\infty B_i$.
In other words, $\Span(B)\subseteq S(K)$  
is that set of points that 
we can obtain from $B$ via successive applications of the tangent and
secant process.  

An {\em Eckardt point} is a point on $S$ where three of the $27$ lines meet.
The main aim of this paper is to prove the following theorems.

\begin{thm}
\label{thm:4+}
Let $K$ be a field with at least $4$ elements. 
Let $S$ be a smooth cubic surface over $K$. 
Suppose $S$ contains a skew pair of lines both defined over $K$. 
Let $P$ be any $K$-rational point on either line 
that is not Eckardt.
Then $\Span(P) = S(K)$.
\end{thm}

\begin{thm}
\label{thm:F3}
Let $K=\F_3$. 
Let $S$ be a smooth cubic surface over $K$. 
Suppose $S$ contains a skew pair of lines $\ell$ and $\ell^\prime$ defined over $K$ 
and that $\ell$ and $\ell^\prime$ each contain at most one $K$-rational Eckardt point.
Then there exists a point $P \in \ell(K)\cup\ell^\prime(K)$ such that
$\Span(P) = S(K)$.
\end{thm}

\begin{thm}
\label{thm:F2}
Let $K=\F_2$. 
Let $S$ be a smooth cubic surface over $K$. 
Suppose $S$ contains a line $\ell$ defined over $K$
that does not contain any $K$-rational Eckardt points.
Then there exists a point $P \in \ell(K)$ such that
$\Span(P) = S(K)$.
\end{thm}

The proofs of Theorems~\ref{thm:F3} and~\ref{thm:F2} use
a certain amount of exhaustive computer enumeration. 
It is perhaps appropriate to add a few words as to
why this is not convenient for Theorem~\ref{thm:4+},
especially as the results of Siksek allow us to
reduce to fields having at most $11$ elements.
To prove Theorem~\ref{thm:4+} by exhaustive enumeration
over a finite field $K$ we need
to enumerate up to projective equivalence
quadruples $(S,\ell,\ell^\prime,P)$ where $S$
is a smooth cubic surface over $K$, $\ell$, $\ell^\prime$
are a skew pair of $K$-lines lying on $S$, and $P \in \ell(K)$
is a non-Eckardt point. Moreover,
for each of these quadruples we would want to apply the tangent
and secant process repeatedly to prove that $\Span(P)=S(K)$.
We did invest some effort into understanding the invariant
theory needed for the enumeration, but it seems to us
that the theory needed to make the enumeration practical
for, say, $K=\F_{11}$, would be far more complicated than our
theoretical proof of Theorem~\ref{thm:4+}.

\section{Preliminary Results}
Here we quote some preliminary results on the geometry and arithmetic of cubic
surfaces.

\begin{thm} (Cayley-Salmon) \label{thm:27}
Every non-singular cubic surface over an algebraically closed field contains exactly $27$ lines.

Every line $\ell$ on the surface meets exactly $10$ other lines,
which break up into $5$ pairs $\ell_i$, $\ell^\prime_i$ ($i=1,\dots,5$) such that
$\ell$, $\ell_i$ and $\ell_i^\prime$ are coplanar,
and $(\ell_i \cup \ell_i^\prime) \cap (\ell_j \cup \ell_j^\prime)=\emptyset$ for $i \ne j$.
\end{thm}
\begin{proof}
For a proof see \cite[V.4]{Hartshorne} or \cite[Section IV.2]{Shaf}.
\end{proof}


For now $S$ will denote a smooth cubic surface in $\PP^3$
over a 
field $K$, defined by a homogeneous cubic polynomial $F \in K[x_0,x_1,x_2,x_3]$.

The remainder of the section summarizes some results (mostly standard)
that can be found in Siksek's paper \cite[Section 2]{Siksek}.
For a point $P \in S(\overline{K})$,
 we shall denote the tangent plane to $S$ at $P$ by $\Pi_P$.
This is given by
$\Pi_P : \nabla{F}(P) \cdot \xx=0$.
We shall write $\Gamma_P$ for the plane curve $S \cdot \Pi_P$. It is
easy to check (using the smoothness of $S$)
that $\Gamma_P$ does not contain any multiple components.
It is a degree $3$ plane curve which is singular at $P$.
If $\Gamma_P$ is irreducible, it is nodal or cuspidal at $P$.
If $\Gamma_P$ is reducible then it is the union of a line and an irreducible conic,
or of three distinct lines.
The curve $\Gamma_P$ contains every $\overline{K}$-line on $S$ that
passes through $P$.

A $\overline{K}$-line $\ell$ is called an {\em asymptotic line} (c.f.\ \cite[Section 2]{Voloch})
at $P \in S(\overline{K})$ if $(\ell \cdot S)_P \geq 3$.
As $S$ is a cubic surface, it is seen that for an asymptotic line $\ell$ at $P$,
either $(\ell \cdot S)_P=3$ or $\ell \subset S$.
The asymptotic lines at $P$ are contained in $\Pi_P$.

Any line contained in $S$ and passing through $P$ is an
asymptotic line through $P$.
The number of distinct asymptotic $\overline{K}$-lines at $P$
is either $1$, $2$ or infinity.
If $S$ has either $1$ or
infinitely many asymptotic lines at $P$ then we shall call $P$ a {\em parabolic} point.
The case where there are infinitely many asymptotic lines  at $P$
is special: in this case $\Gamma_P$
decomposes as a union of three $\overline{K}$-lines passing through $P$ lying on $S$
and so the point
$P$ is an {\em Eckardt} point. If $P$ is parabolic
but not Eckardt, the curve $\Gamma_P$ has a cusp
at $P$. Note that for $P$ lying on a line $\ell\subset S$, if $P$ is not Eckardt then
 $\Gamma_P=\ell\cup C$ where $C$ is an irreducible conic, and $\ell$ lies tangent to $C$.
 If $P$ is non-parabolic, then $\Gamma_P$ has a node
at $P$.

We shall also need to study the number of parabolic points on a
line lying on a cubic surface. Let ${\PP^3}^*$ be the dual projective
space and write $\gamma : S \rightarrow {\PP^3}^*$ for the
{\em Gauss map} which sends a point to its tangent plane.
A useful characterisation of parabolic points is that they
are the points of ramification of the Gauss map \cite[Section 2]{Voloch}.
If $\ell \subset S$ and $P \in \ell$, then $\ell$ is contained
in the tangent plane $\Pi_P$. The family of planes through $\ell$ can be
identified with $\PP^1$ and once such an identification is fixed we let
$\gamma_\ell : \ell \rightarrow \PP^1$ be the map that sends a point
on $\ell$ to its tangent plane through $\ell$. The map $\gamma_\ell$ has
degree $2$ (\cite[proof of Lemma 2.2]{Siksek}), and hence is separable
if $\ch(K) \ne 2$.

\begin{mylem}\label{lem:paraline} (Siksek \cite[Lemma 2.2]{Siksek})
Let $\ell$ be a $K$-line contained in $S$. 
\begin{enumerate}
\item[(i)] If $\ch(K) \neq 2$ then $\gamma_\ell$ is separable.
Precisely two points
$P \in \ell(\overline{K})$
are parabolic, and so there are at most two Eckardt points on $\ell$.
\item[(ii)] If $\ch(K)=2$ and $\gamma_\ell$ is separable then
there is precisely one point $P \in \ell(\overline{K})$
which is parabolic and so at most one Eckardt point on $\ell$.
\item[(iii)] If $\ch(K)=2$ and $\gamma_\ell$ is inseparable then
every point $P \in \ell(\overline{K})$ is parabolic
and the line $\ell$ contains exactly $5$ Eckardt points.
\end{enumerate}
\end{mylem}

Finally we shall need the following result, which in effect says that we can restrict ourselves
to fields having at most $11$ elements in the proofs of Theorems~\ref{thm:4+}--\ref{thm:F2}.
\begin{thm} (Siksek \cite[Theorem 1]{Siksek}) \label{thm:siksek}
Let $K$ be a field with at least $13$ elements. Let $S$ be a smooth cubic surface over $K$. Suppose $S$ contains a pair of skew lines both defined over $K$. Let $P \in S(K)$ be a point on either line that is not an Eckardt point. Then $Span(P) =
S(K)$.
\end{thm}

\section{Proof of Theorem~\ref{thm:4+}}

\begin{mylem}
\label{lema:3.1,4+}
Let $K$ be a field with at least $4$ elements and $S$ a smooth cubic surface defined over $K$. Let $\ell$ 
be a $K$-line on $S$. Let $P \in \ell(K)$ be a point that does not lie on
any other line belonging to $S$. Then
\[
\ell(K) \subseteq \Gamma_P(K) \subseteq \Span(P).
\]
\end{mylem}
\begin{proof}
The curve $\Gamma_P$ has degree $3$. The line $\ell$ is an irreducible component of $\Gamma_P$,
and there are no other lines in $S$ passing through $P$. 
Thus $\Gamma_P= \ell \cup C$, where $C$ is an irreducible conic. 
Note $C \cdot \ell = P+P^\prime$ where $P^\prime$ is a point in $S(K)$.
Note also that since $P^\prime$ lies on both $\ell$ and $C$, any line
$\ell^\prime \subset \Pi_P$, $\ell^\prime \neq \ell$ going through $P^\prime$
will have a double intersection at $P^\prime$, i.e. $\Pi_P$ is the tangent
plane at $P^\prime$ and hence $\Gamma_{P^\prime} = \Gamma_P$.
We want to show that $\Gamma_P(K) \subseteq \Span(P)$.

Let $Q \in C(K)$, $Q \ne P, P^\prime$.
Let $\ell^\prime$ be the line joining $P$ and $Q$. Then $\ell^\prime \cdot S=2P+Q$. Thus $Q \in \Span(P)$.
Hence $C(K) \backslash \{P^\prime\} \subseteq \Span(P)$. We now want to 
show that $\ell(K) \backslash \{P^\prime\} \subseteq \Span(P)$. Fix $Q \in C(K)\backslash\{P,P^\prime\}$, let 
$R \in \ell(K)\backslash\{P,P^\prime\}$ and let $\ell^\prime$ be
the line joining $Q$ and $R$. Then $\ell^\prime \cdot S=Q+R+R^\prime$, where $R^\prime \in C(K)$.
Since $Q$, $R^\prime \in C(K)\backslash\{P^\prime\} \subseteq \Span(P)$, we have $R \in \Span(P)$. Thus
$\Gamma_P(K)\backslash\{P^\prime\} \subseteq \Span(P)$.

If $P=P^\prime$, then we have $\Gamma_P(K)\subset\Span(P)$ and we are done. 
Suppose now that $P \ne P^\prime$. To complete the proof,
we must show that $P^\prime \in \Span(P)$. 
As $P \ne P^\prime$ but $\Gamma_P=\Gamma_{P^\prime}$, it follows from Lemma~\ref{lem:paraline} that
 $\gamma_\ell$ 
is separable, and that therefore the line $\ell$ contains at most
two Eckardt points. Since $\lvert K \rvert \geq 4$, the line $\ell$ has at least
five $K$-rational points, and so there is some point $R \in \ell(K)$ that 
is neither Eckardt, nor equal to $P$, $P^\prime$. As noted above $\Pi_P=\Pi_{P^\prime} \supset \ell\cup C$.
As $\gamma_{\ell}$ has degree $2$, we see that $\Pi_R \ne \Pi_P$. 
There are now two cases to consider. The first is when $\Gamma_R=\ell \cup C^\prime$
where $C^\prime$ is an irreducible conic, and the second is when $\Gamma_R$ is a union
of three lines. For the first case, we have 
\[
\Gamma_R(K) \backslash \{R^\prime\} \subseteq 
\Span(R) \subseteq \Span(P)
\]
where $\ell \cdot C^\prime=R+R^\prime$. Note, $P^\prime \neq R^\prime$,
as $\Pi_{P^\prime}=\Pi_P \ne \Pi_R=\Pi_{R^\prime}$. Hence $P^\prime \in \Span(P)$,
and the proof is complete in this case.

Finally, we must consider the case where $\Gamma_R$ is the union of three lines, which 
must include $\ell$. Let the other two lines be $\ell_2$ and $\ell_3$, where $\ell \cdot \ell_2=R$,
$\ell \cdot \ell_3=R^\prime$ and $\ell_2 \cdot \ell_3=R^{\prime\prime}$. As $R$
is not Eckardt, $R$, $R^\prime$ and $R^{\prime \prime}$ are distinct. Since $\ell$ and $R$ 
are $K$-rational, so are the lines $\ell_2$, $\ell_3$ and the points $R^\prime$
and $R^{\prime\prime}$.

First note that both $R$ and $R^\prime$ are in $\Span(P)$ as they both lie on $\ell$ and are not equal to $P^\prime$. 
Let $Q \in \ell_3(K)$, $Q \ne R^\prime$, $R^{\prime\prime}$. Let $m$ be the line joining $R$ and $Q$.
Then $m\cdot S=2R+Q$, and so $Q \in \Span(P)$. Thus
\[
\ell_3(K) \backslash \{R^{\prime\prime}\} \subseteq \Span(P)
\]
and likewise
\[
\ell_2(K) \backslash \{R^{\prime\prime}\} \subseteq \Span(P).
\]
Take $Q \in \ell_3(K)$, $Q \ne R^\prime$, $R^{\prime\prime}$. Let $m$ be the line joining $P^\prime$ and $Q$.
Then $m \cdot S=P^\prime+Q+Q^\prime$, where $Q^\prime \in \ell_2(K)$, $Q^\prime \ne R$, $R^{\prime\prime}$.
Thus $P^\prime \in \Span(P)$, completing the proof.
\end{proof}

The following lemma is a strengthening of Lemma~\ref{lema:3.1,4+}.
\begin{mylem}
\label{lema:3.2,4+}
Let $K$ be a field with at least 4 elements and $S$ a smooth cubic surface defined over $K$. 
Let $\ell$ be a $K$-line on $S$. Let $P \in \ell(K)$ and suppose that $P$ is not an
Eckardt point. Then
\[
\ell(K) \subseteq \Gamma_P(K) \subseteq \Span(P).
\]
\end{mylem}
\begin{proof}
Let $P \in \ell(K)$ be a non-Eckardt point. 
If $P$ does not lie on any other line contained in $S$ then we can invoke Lemma~\ref{lema:3.1,4+}.
Thus we may suppose that $P$ lies on some other line $\ell_2$. This is necessarily a $K$-line
because if it were not, its conjugate line would also pass through $P$, meaning that $P$ were
an Eckardt point, which would contradict the hypotheses of the lemma. Now $\Gamma_P=\ell \cup \ell_2 \cup \ell_3$,
where $\ell_3$ is also a $K$-line. As $P$ is not Eckardt, $\ell_3$ does not pass through $P$.
Let $\ell \cdot \ell_3=P^\prime$ and $\ell_2 \cdot \ell_3=P^{\prime\prime}$.  As before 
\[
\ell_3(K)\backslash \{P^\prime,P^{\prime\prime}\} \subseteq \Span(P).
\]
As in the proof of Lemma~\ref{lema:3.1,4+}, $\gamma_{\ell_3}$ is separable, and so by Lemma~\ref{lem:paraline},
there are at most two Eckardt points on $\ell_3$. Since $\ell_3(K)$ has at least $5$ points,
we see that there is some $Q \in \ell_3(K)\backslash \{P^\prime, P^{\prime\prime}\}$ that is not Eckardt.
We consider two cases. The first is where $Q$ does not lie on any other line. Then, by Lemma~\ref{lema:3.1,4+},
$\ell_3(K) \subseteq \Span(Q) \subseteq \Span(P)$. Thus $P$, $P^\prime$, $P^{\prime\prime} \in \Span(P)$.
As before, we can generate the remaining points in $\Gamma_P(K)=\ell(K) \cup \ell_2(K) \cup \ell_3(K)$,
from these.

The remaining case is when $Q$ lies on some other line $\ell_4$ and so $\Gamma_Q=\ell_3 \cup \ell_4 \cup \ell_5$.
Just as in the argument at the end of the proof of Lemma~\ref{lema:3.1,4+}, we can show that $P$, $P^\prime$
and $P^{\prime\prime}$ are in $\Span(P)$ and complete the proof.
\end{proof}

The following propositions will be useful.

\begin{prop}
\label{prop:skewlines}
Let $K$ be a field with at least $4$ elements and 
$S$ be a smooth cubic surface defined over $K$.
Suppose $S$ contains a skew pair of lines $\ell$, $\ell^\prime$ 
and there is a non-Eckardt point $P \in \ell(K)$ such that 
$\Pi_P\cdot\ell^\prime$ is also non-Eckardt.
Then
\[
 \ell^\prime(K) \subseteq \Span(\ell(K)).
\]
\end{prop}
\begin{proof}
 Let $Q=\Pi_P\cdot\ell^\prime$. Note that $Q \in \Gamma_P(K)$, thus $Q \in \Span(P) \subseteq \Span(\ell(K))$ 
by Lemma~\ref{lema:3.2,4+}. Suppose $Q$ is non-Eckardt. Applying 
Lemma~\ref{lema:3.2,4+} again we have
\[
 \ell^\prime(K) \subseteq \Span(Q) \subseteq \Span(\ell(K)).
\]
\end{proof}

\begin{prop}
\label{prop:conic}
Let $S$ be a smooth cubic surface defined over a field $K$. 
Suppose $S$ contains a skew pair of $K$-lines $\ell$, $\ell^\prime$ 
and let $P \in \ell(K)$. 
Let $\Gamma_P$ be the union of $\ell$ and an irreducible conic.
Then the point $\ell^\prime \cdot \Pi_P$ is not an Eckardt point.
\end{prop}
\begin{proof}
Let $P \in \ell(K)$.
Let $E =\ell^\prime \cdot \Pi_P$. 
Suppose $\Gamma_P=C \cup \ell$ where $C$ is an irreducible conic. 
We will first show that as $C$ is irreducible, it must be absolutely irreducible. 
Indeed, suppose $C=m \cup n$ where $m$, $n$ are $\overline{K}$-lines, that
are Galois conjugates. The $K$-point $P$ lies on one of them and hence both.
 Now
$E$ is a $K$-point, and $E$ belongs to $\Gamma_P=\ell \cup m \cup n$
as well as $\ell^\prime$. Since $\ell$ and $\ell^\prime$ are skew, without loss of generality
$E\in m$. Hence the line $m$ joins the $K$-points $P$ and $E$ and is therefore
defined over $K$. This contradicts the irreducibility of $C$. 
Hence $C$ is an absolutely irreducible conic.

We must prove that $E$ is not an Eckardt point. Suppose it is.
Let $\ell_2$ and $\ell_3$ be the other two lines going through $E$. 
Let $Q=\ell \cdot \Pi_E$, then without loss of generality $Q = \ell \cdot \ell_2$. 
Note that $\ell_2$ must be in the tangent plane to $S$ at $Q$, so $\Pi_Q \neq \Pi_P$ since $\Gamma_P=\ell \cup C$. 
Note also that $E=\ell^\prime \cdot \Pi_Q$, which implies that $\Pi_Q$ is the unique plane containing $\ell$ and $E$. 
However, the plane $\Pi_P$ also contains $\ell$ and $E$. 
But $\Pi_P \neq \Pi_Q$ so we have reached a contradiction, and the point $\ell^\prime\cdot\Pi_P$ cannot be an Eckardt point.
\end{proof}

\begin{mylem}
\label{lema:3.3,7+}
Let $K$ be a field with at least 7 elements and $\ch(K) \neq 2$. 
Let $S$ be a smooth cubic surface defined over $K$. 
Suppose $S$ contains a skew pair of $K$-lines $\ell$, $\ell^\prime$. 
Then
\[
\ell^\prime(K) \subseteq \Span(\ell(K)).
\]
\end{mylem}
\begin{proof}
By Lemma~\ref{lem:paraline}, there are at most two $K$-rational Eckardt points on each of $\ell$, $\ell^\prime$.
Hence
\[
\#(\ell(K)\setminus \{\text{Eckardt points}\}) \geq 6.
\]
The Gauss map on $\ell$ has degree 2 so 
\[
\#\gamma_\ell(\ell(K)\setminus \{\text{Eckardt points}\}) \geq 3.
\]
Therefore we must have a non-Eckardt $P \in\ell(K)$ mapping to a plane $\gamma_\ell(P)$ that intersects 
$\ell^\prime$ in a non-Eckardt point $Q$. We invoke Proposition~\ref{prop:skewlines} to obtain $\ell^\prime(K) \subseteq \Span(\ell(K))$,
which completes the proof.
\end{proof}

\begin{mylem}
\label{lema:3.3,F4F5F8}
 Let $K$ be $\F_4$, $\F_5$ or $\F_8$ and $S$ be a smooth cubic surface defined over $K$.
Suppose $S$ contains a skew pair of $K$-lines $\ell$, $\ell^\prime$.
Let $P\in\ell(K)$ be a point that is not Eckardt.
Then
\[
 \ell_2(K) \subseteq \Span(\ell_1(K)),
\]
where $\ell_1$, $\ell_2$ is a skew pair of $K$-lines in $S$ that may or may not be equal to $\ell$, $\ell^\prime$.
\end{mylem}
\begin{proof}
 First note that for any non-Eckardt point $P\in\ell(K)$ we have $\ell(K)\subseteq\Span(P)$ by Lemma~\ref{lema:3.2,4+}. 
Suppose $P\in\ell(K)$ is a point such that $\Gamma_P=C\cup\ell$ where $C$ is an irreducible conic. 
Then $Q=\Pi_P\cdot\ell^\prime$ is a non-Eckardt point by Proposition~\ref{prop:conic}. We invoke Lemma~\ref{lema:3.2,4+}
to obtain
\[
 \ell^\prime(K) \subseteq \Span(Q) \subseteq \Span(P) \subseteq \Span(\ell(K)).
\]

Therefore we assume that all points in $\ell(K)$ lie on at least one $K$-line in $S$ other than $\ell$.
Note that this excludes the case where $\ch(K)=2$ and $\gamma_\ell$ is inseparable, since in such cases
we must have at least one non-Eckardt point in $\ell(K)$ and by Lemma~\ref{lem:paraline} any such $P$ must be parabolic, 
and hence $\Gamma_P$ is the union of $\ell$ and an irreducible conic.
Thus we may assume that $\gamma_\ell$ is separable.
In which case we must have four points $P$, $P^\prime$, $R$, $R^\prime \in \ell(K)$ with
$\Gamma_P=\Gamma_{P^\prime}$, $\Gamma_R=\Gamma_{R^\prime}$. Let $\ell_1\subseteq S$ be the $K$-line such that 
$\ell\cdot\ell_1=P$ and $\ell_2\subseteq S$ be the $K$-line such that $\ell\cdot\ell_2=R$. Then $l_1\subseteq\Pi_P$ and
$\ell_2\subseteq\Pi_R$ and so $\ell_1$ is skew to $\ell_2$. Note that $P$ is a non-Eckardt point on $\ell_1$ so
$\ell_1(K) \subseteq \Span(P)$ and likewise $\ell_2(K) \subseteq \Span(R) \subseteq \Span(P)$.
Hence
\[
 \ell_1(K)\cup\ell_2(K) \subseteq \Span(P),
\]
which completes the proof.
\end{proof}

The following lemma is stated in \cite{Siksek} with the hypothesis that $K$ has at least $13$ elements. 
Using our Lemma~\ref{lema:3.2,4+} and by modifying the proof we can strengthen this as follows.

\begin{mylem}\label{lema:3.4,4+}
Let $K$ be a field with at least $4$ elements, and let $S$ be a smooth cubic surface
defined over $K$. Suppose $S$ contains a pair of skew lines $\ell_1$ and $\ell_2$ both defined over $K$. 
If $\# K=4$, then suppose also that at least one of them contains a non-Eckardt $K$-point.
Then
\[
\Span(\ell_1(K)\cup \ell_2(K)) = S(K).
\]
\end{mylem}
\begin{proof}
Let $P$ be a $K$-point on $S$ not belonging to either line;
we will show that $P$ belongs to the span of $\ell_1(K) \cup \ell_2(K)$.
Let $\Pi_1$ be the unique plane containing $\ell_2$ and $P$,
and $\Pi_2$ the unique plane containing $\ell_1$ and $P$.
Since $\ell_1$ and $\ell_2$ are skew we know that $\ell_i \not \subset \Pi_i$.
Write $Q_i=\ell_i \cap \Pi_i$.
Note that $P$, $Q_1$ and $Q_2$ are distinct points on $S$ that also belong
to the $K$-line $\ell=\Pi_1 \cap \Pi_2$. Suppose first that $\ell \not \subset S$. Then $\ell\cdot S=P+Q_1+Q_2$. 
Thus $P \in \Span\left(\ell_1(K)\cup \ell_2(K) \right)$ as required.

Next suppose that $\ell \subset S$. 
If $\lvert K \rvert \geq 5$, we know from Lemma~\ref{lem:paraline} that there are non-Eckardt 
$K$-points on both $\ell_1$ and $\ell_2$. If $\lvert K \rvert=4$, then one of the hypotheses
of the lemma is that there is a non-Eckardt $K$-point on one of those two lines. Without loss of generality,
$R \in \ell_2(K)$ is non-Eckardt.

Now $\ell \subset \Gamma_{Q_1}$.
If $Q_1$ is not Eckardt, then by
Lemma~\ref{lema:3.2,4+},
\[
P \in \ell(K)  \subseteq \Gamma_{Q_1}(K) \subseteq \Span(Q_1) \subseteq \Span(\ell_1(K)).
\]
Thus we may assume that $Q_1$ is Eckardt. Similarly $Q_2$ is Eckardt.
Then $\Gamma_{Q_1}=\ell \cup \ell_1 \cup \ell_3$ where $\ell_3$ is also $K$-rational. 
Now $\ell_2$ must meet that tangent plane $\Pi_{Q_1}$ in a unique point, and that is $Q_2 \in \ell$.
Therefore, $\ell_2$ and $\ell_3$ are skew. Consider $\gamma_{\ell_2}$. As $Q_2$ is Eckardt,
it is a ramification point for $\gamma_{\ell_2}$. Therefore $\gamma_{\ell_2}(Q_2) \ne \gamma_{\ell_2}(R)$.
Note $\gamma_{\ell_2}(Q_2) \cdot \ell_3=Q_1$, so $\gamma_{\ell_2}(R) \cdot \ell_3=R^\prime$
where $R^\prime$ is a $K$-point distinct from $Q_1$. Moreover, $R^\prime \in \Span(R) \subseteq \Span(\ell_2(K))$.
Finally, consider the line that joints $R^\prime$ and $P$. This lies in $\Pi_{Q_1}$, but not on $S$, and so must intersect
$\ell_1$ in a $K$-point $R^{\prime\prime}$. Hence $P \in \Span(\ell_1(K) \cup \ell_2(K))$
which completes the proof.
\end{proof}

\begin{proof}[Proof of Theorem~\ref{thm:4+}]
If $K$ has $13$ or more elements then we can invoke Siksek's Theorem~\ref{thm:siksek}. Thus we may
restrict our attention to the cases $\#K=4$, $5$, $7$, $8$, $9$, $11$.
The proof follows from Lemmas~\ref{lema:3.2,4+}, \ref{lema:3.3,7+}, \ref{lema:3.3,F4F5F8} and \ref{lema:3.4,4+}.
\end{proof}

\section{Proof of Theorem~\ref{thm:F3}}

\begin{mylem}
\label{lema:F3compute}
Let $K=\F_3$ and $S$ be a smooth cubic surface defined over $K$.
Suppose $S$ contains a skew pair of $K$-lines $\ell$, $\ell^\prime$
and suppose $\ell$ contains exactly one $K$-rational Eckardt point.
Then there exists a non-Eckardt point $P\in\ell(K)$ such that
\[
 S(K)=\Span(P).
\]
\end{mylem}
\begin{proof}
The lemma was proved by an exhaustive computer enumeration implemented in {\tt MAGMA} \cite{MAGMA}.
By a projective change of co-ordinates we may first suppose that the line $\ell$ is defined by $X=Y=0$,
and that therefore the surface $S$ has the form $X Q_1+Y Q_2$ where $Q_1\in\F_3[X,Y,Z,W]$ and
$Q_2 \in \F_3[Y,Z,W]$ are homogeneous quadratic forms. We may then by further projective changes of co-ordinates
suppose that the $K$-rational Eckardt point on $\ell$ is the point $P=(0:0:0:1)$.
We denote the other two lines in $S$ passing through $P$ by $\ell_1$, $\ell_2$. 
The line $\ell^\prime$ must intersect the plane $\Pi_P$ in some $K$-point $Q \ne P$.
As $\ell$ and $\ell^\prime$ are skew, $Q \notin \ell$ and so without loss of generality
$Q \in \ell_1$. The line $\ell_1$ joins two $K$-points and is therefore a $K$-line. Hence
$\ell_2$ is also a $K$-line.
By yet another change of coordinates that preserves $\ell$ and $P$, 
we may suppose that $\ell_1$ and $\ell_2$ have the equations $\ell_1:X=Z=0$ and $\ell_2:X=Y+Z=0$. 
Thus every cubic surface defined over
$\F_3$ containing a skew pair of $K$-lines has a model that can be written in the form
\[
 X(aX^2+bXY+cXZ+dY^2+eY+fZ^2+gW^2)+YZ(Y+Z)
\]
where $a,\dots , g \in \F_3$. Therefore the program was enumerated over $3^7=2187$ cubic surfaces.
Our program checked the surfaces for smoothness, then whether there was a point $P\in \ell$ such that
$\Span(P)=S(K)$. In the cases where this failed, we verified that there was a second Eckardt point in $\ell(K)$.
%
\end{proof}

The remaining cases in which there are no $K$-rational Eckardt points on either $\ell$, $\ell^\prime$ 
result from the following lemmas.

\begin{mylem}
\label{lema:3.1,F3}
Let $K=\F_3$ 
and $S$ be a smooth cubic surface defined over $K$. 
Let $\ell$ be a $K$-line on $S$ that does not contain any $K$-rational Eckardt points. 
Let $P \in \ell(K)$ be a point that does not lie on any other line belonging to $S$. 
Then
\[
\ell(K) \subseteq \Gamma_P(K) \subseteq \Span(P).
\] 
\end{mylem}
\begin{proof}
If $P$ is a parabolic point then, as in the proof of Lemma~\ref{lema:3.1,4+},
we know that
$\Gamma_P(K)\subseteq\Span(P)$.
Otherwise there is a point $P^\prime\in\ell(K)$ such that $P^\prime \ne P$ 
but $\Gamma_{P^\prime}=\Gamma_P$.
In this case, similarly to the proof of Lemma~\ref{lema:3.1,4+}, we have
\[
\Gamma_P(K) \setminus \{P^\prime\} \subseteq \Span(P).
\]
As $K=\F_3$ there are four points in $\ell(K)$, $P$, $P^\prime$, $R$ and $R^\prime$.
If $\Gamma_R=\ell \cup C$ where $C$ is an irreducible conic then
$P^\prime\in\Span(R)\subseteq\Span(P)$. Hence $\Gamma_P(K)\subseteq\Span(P)$.
Otherwise $\Gamma_R=\Gamma_{R^\prime}$ and is the union of $3$ $K$-lines in $S$,
which are 
$\ell$, $\ell_2$ and $\ell_3$, where $\ell\cdot\ell_2=R$, $\ell\cdot\ell_3=R^\prime$ 
and $\ell_2\cdot\ell_3=R^{\prime\prime}$. We know that 
$(\ell_2(K)\cup\ell_3(K))\setminus\{R^{\prime\prime}\}\subseteq\Span(R,R^\prime)$
and $P^\prime\in\Span((\ell_2(K)\cup\ell_3(K)\setminus\{R^{\prime\prime}\})$.
Thus
\[
\ell(K)\subseteq\Gamma_P(K)\subseteq\Span(P)
\]
which completes the proof.
\end{proof}

\begin{mylem}
\label{lema:3.2,3}
Let $K=\mathbb{F}_3$ and $S$ be a smooth cubic surface defined over $K$. 
Suppose $S$ contains a skew pair of $K$-lines, $\ell$, $\ell^\prime$ that contains no $K$-rational Eckardt points.
Then there is a point $P \in \ell(K)$ such that 
\[
\ell(K)\subseteq\Gamma_P\subseteq\Span(P).
\]
\end{mylem}
\begin{proof}
 If there is a point in $\ell(K)$ that lies on no other line in $S$ 
then the result follows from Lemma~\ref{lema:3.1,F3}.
So suppose every point in $\ell(K)$ lies on exactly one other $K$-line in $S$.
Since $K=\F_3$ there are $4$ points in $\ell(K)$, which we denote $P$, $P^\prime$, $R$, $R^\prime$.
We have $\Gamma_P=\Gamma_{P^\prime}=\ell\cup\ell_1\cup\ell_2$ with $P=\ell\cdot\ell_1$, $P^\prime=\ell\cdot\ell_2$
and $\Gamma_R=\Gamma_{R^\prime}$ with $R=\ell\cdot\ell_3$, $R^\prime=\ell\cdot\ell_4$. Let 
$P^{\prime\prime}=\ell_1\cdot\ell_2$. By the argument in the proof of Lemma~\ref{lema:F3compute} 
we know that $\ell^\prime$ intersects precisely 
one of $\ell_1$, $\ell_2$ and one of $\ell_3$, $\ell_4$. Without loss of generality suppose that $\ell^\prime$ intersects 
$\ell_2$ and $\ell_4$. By our hypotheses $\ell^\prime$ contains no $K$-rational Eckardt points, therefore the point $Q=\ell^\prime\cdot\ell_2$ is non-Eckardt. Note that $Q\in(\ell_2(K)\setminus\{P^\prime, P^{\prime\prime}\}) \subseteq \Span(P)$. Let $Q^\prime$ be the remaining point in $\ell_2(K)\setminus\{P^\prime,P^{\prime\prime}\}$.
Our aim is the show that $P^\prime$, $P^{\prime\prime} \in \Span(P)$ since we can generate all the remaining points in 
$\ell(K)$, $\ell_1(K)$ from $P^{\prime\prime}$, $P^\prime$ respectively. If $Q$ is a parabolic point then 
\[
P^\prime, P^{\prime\prime}\in \ell_2(K) \subseteq \Span(Q) \subseteq \Span(P).
\]
Likewise if $\Gamma_Q=\Gamma_{Q^\prime}$ then
\[
 P^\prime, P^{\prime\prime} \in \ell_2(K) \subseteq \Span(Q,Q^\prime) \subseteq \Span(P),
\]
which completes the proof.
\end{proof}

\begin{mylem}
\label{lema:3.3,3}
Let $K=\F_3$ and $S$ be a smooth cubic surface defined over $K$. 
Suppose $S$ contains a skew pair of $K$-lines $\ell$, $\ell^\prime$ that contains no 
$K$-rational Eckardt points. 
Then
\[
 \ell^\prime(K) \subseteq \Span(\ell(K)).
\]
\end{mylem}
\begin{proof}
Let $P\in \ell(K)$. By Lemma~\ref{lema:3.2,3} $Q=\Pi_P\cdot\ell^\prime \in \Gamma_P \subseteq \Span(P)$.
We invoke Lemma~\ref{lema:3.2,3} again to obtain
\[
 \ell^\prime(K) \subseteq \Span(Q) \subseteq \Span(P) \subseteq \Span(\ell(K)).
\]
\end{proof}

\begin{mylem}
\label{lema:3.4,3}
Let $K=\F_3$, and let $S$ be a smooth cubic surface
defined over $K$. Suppose $S$ contains a skew pair of $K$-lines $\ell_1$, $\ell_2$, which contains no
$K$-rational Eckardt points.
Then
\[
\Span(\ell_1(K)\cup \ell_2(K)) = S(K).
\]
\end{mylem}
\begin{proof}
This proof is similar to the proof of Lemma~\ref{lema:3.4,4+}. 
Let $P$ be a $K$-point on $S$ not belonging to either line;
we will show that $P$ belongs to the span of $\ell_1(K) \cup \ell_2(K)$.
Let $\Pi_1$ be the unique plane containing $\ell_2$ and $P$,
and $\Pi_2$ the unique plane containing $\ell_1$ and $P$.
Since $\ell_1$ and $\ell_2$ are skew we know that $\ell_i \not \subset \Pi_i$.
Write $Q_i=\ell_i \cap \Pi_i$.
Note that $P$, $Q_1$ and $Q_2$ are distinct points on $S$ that also belong
to the $K$-line $\ell=\Pi_1 \cap \Pi_2$. Suppose first that $\ell \not \subset S$. Then $\ell\cdot S=P+Q_1+Q_2$. 
Thus $P \in \Span\left(\ell_1(K)\cup \ell_2(K) \right)$ as required.

Next suppose that $\ell \subset S$. Now $\ell \subset \Gamma_{Q_1}$.
Since $Q_1$ is not Eckardt, then by
Lemma~\ref{lema:3.2,3},
\[
P \in \ell(K)  \subseteq \Gamma_{Q_1}(K) \subseteq \Span(Q_1) \subseteq \Span(\ell_1(K)),
\]
which completes the proof.
\end{proof}

\begin{proof}[Proof of Theorem~\ref{thm:F3}]
The proof follows from Lemmas~\ref{lema:3.2,3}, \ref{lema:3.3,3}, \ref{lema:3.4,3} and \ref{lema:F3compute}.
\end{proof}

\section{Proof of Theorem~\ref{thm:F2}}

Theorems~\ref{thm:F2} was proved by an exhaustive  computer enumeration implemented in {\tt MAGMA} 
\cite{MAGMA}. 

\begin{proof}[Proof of Theorem~\ref{thm:F2}] By a projective change of coordinates
we may suppose that the line $\ell$ is defined by $X=Y=0$, and  that therefore the surface
$S$ has the form $X Q_1+Y Q_2$ where $Q_1 \in \F_2[X,Y,Z,W]$ and $Q_2 \in \F_2[Y,Z,W]$ are
homogeneous quadratic forms. Our program enumerated all possible $Q_1$, $Q_2$, checked the surface
for smoothness and whenever $\ell$ contained no $K$-rational Eckardt points, it verified that the span of one
of its $K$-points is equal to $S(K)$. This meant the program was enumerated over $2^{16}=65536$ possible models,
as there are $10$ monomials in $X$, $Y$, $Z$, $W$, and $6$ monomials in $Y$, $Z$, $W$.
\end{proof}

\end{document}